\documentclass[reqno,centertags, 12pt]{amsart}
\usepackage{amsmath,amsthm,amscd,amssymb}
\usepackage{latexsym}
\usepackage{graphicx}
\usepackage[raiselinks,colorlinks]{hyperref}
\hypersetup{citecolor=blue}
\newcommand{\bbC}{{\mathbb{C}}}
\newcommand{\bbD}{{\mathbb{D}}}

\newcommand{\bbF}{{\mathbb{F}}}

\newcommand{\bbR}{{\mathbb{R}}}

\newcommand{\bbZ}{{\mathbb{Z}}}

\newcommand{\calE}{{\frak{e}}}

\newcommand{\calR}{{\mathcal R}}

\newcommand{\calT}{{\mathcal T}}

\newcommand{\dott}{\,\cdot\,}

\newcommand{\lb}{\label}
\newcommand{\f}{\frac}

\newcommand{\ol}{\overline}

\newcommand{\dist}{\text{\rm{dist}}}

\newcommand{\id}{\text{\rm{id}}}

\newcommand{\ess}{\text{\rm{ess}}}
\newcommand{\ac}{\text{\rm{ac}}}

\newcommand{\s}{\text{\rm{s}}}

\newcommand{\bi}{\bibitem}

\newcommand{\beq}{\begin{equation}}
\newcommand{\eeq}{\end{equation}}
\newcommand{\ba}{\begin{align}}
\newcommand{\ea}{\end{align}}
\newcommand{\veps}{\varepsilon}

\newcounter{smalllist}
\newenvironment{SL}{\begin{list}{{\rm\roman{smalllist})}}{%
\setlength{\topsep}{0mm}\setlength{\parsep}{0mm}\setlength{\itemsep}{0mm}%
\setlength{\labelwidth}{2em}\setlength{\leftmargin}{2em}\usecounter{smalllist}%
}}{\end{list}}

\DeclareMathOperator{\Ima}{Im}

\allowdisplaybreaks
\numberwithin{equation}{section}

\newtheorem{theorem}{Theorem}[section]

\newtheorem*{p2.1}{Proposition 2.1}
\newtheorem{proposition}[theorem]{Proposition}

\newtheorem{corollary}[theorem]{Corollary}
\theoremstyle{definition}

\newtheorem{conjecture}[theorem]{Conjecture}

\theoremstyle{remark}
\newtheorem*{remark}{Remark}
\newtheorem*{remarks}{Remarks}

\newcommand{\abs}[1]{\lvert#1\rvert}

\renewcommand{\MRhref}[2]{\href{http://www.ams.org/mathscinet-getitem?mr=#1}{#2}}
\renewcommand{\MR}[1]{}

\makeatletter
\def\@strippedMR{}
\def\@scanforMR#1#2#3\endscan{%
   \ifx#1M\ifx#2R\def\@strippedMR{#3}%
   \else\def\@strippedMR{#1#2#3}%
   \fi\fi}
\renewcommand\MR[1]{\relax\ifhmode\unskip\spacefactor3000 \space\fi
   \@scanforMR#1\endscan
   \MRhref{\@strippedMR}{MR\@strippedMR}}
\makeatother

\begin{document}
\title[Finite Gap Jacobi Matrices: An Announcement]{Finite Gap Jacobi Matrices:\\An Announcement}
\author[J.~S.~Christiansen, B.~Simon, and M.~Zinchenko]{Jacob S.~Christiansen$^*$, Barry Simon$^{*,\dagger}$, and
Maxim Zinchenko$^*$}

\thanks{$^*$ Mathematics 253-37, California Institute of Technology, Pasadena, CA 91125.
E-mail: stordal@caltech.edu; bsimon@caltech.edu; maxim@caltech.edu}
\thanks{$^\dagger$ Supported in part by NSF grants DMS--0140592
and DMS-0652919}

\date{October 22, 2007}

\begin{abstract} We consider Jacobi matrices whose essential spectrum is a finite union of closed intervals.
We focus on Szeg\H{o}'s theorem, Jost solutions, and Szeg\H{o} asymptotics for this situation. This announcement
describes talks the authors gave at OPSFA 2007.
\end{abstract}

\maketitle

\section{Introduction and Background} \lb{s1}

This paper announces results in the spectral theory of orthogonal polynomials on the real line (OPRL).
We start out with a measure $d\mu$ of compact support on $\bbR$; $P_n(x;d\mu)$ (sometimes we drop $d\mu$) and $p_n(x;d\mu)$ are the monic orthogonal and orthonormal polynomials, and $\{a_n,b_n\}_{n=1}^\infty$
the Jacobi parameters determined by the recursion relations (where $p_{-1}=0$):
\begin{equation} \lb{1.1}
xp_n(x)=a_{n+1} p_{n+1}(x) + b_{n+1} p_n(x) + a_n p_{n-1}(x)
\end{equation}
summarized in a Jacobi matrix
\begin{equation} \lb{1.2}
J=
\begin{pmatrix}
b_1 & a_1 & 0 & 0 & \cdots \\
a_1 & b_2 & a_2 & 0 & \cdots \\
0 & a_2 & b_3 & a_3 & \cdots \\
\vdots & \vdots & \vdots & \vdots & \ddots
\end{pmatrix}
\end{equation}
We will use the Lebesgue decomposition of $d\mu$,
\begin{equation} \lb{1.3}
d\mu(x)=w(x)\, dx + d\mu_\s(x)
\end{equation}
with $d\mu_\s$ singular w.r.t.\ $dx$.

In this introduction, we will also consider orthogonal polynomials
on the unit circle (OPUC) where $d\mu$ is now a measure on
$\partial\bbD =\{e^{i\theta}\mid\theta\in [0,2\pi)\}$;
$\Phi_n(z;d\mu)$ and $\varphi_n(z;d\mu)$ are the monic orthogonal
and orthonormal polynomials, and
\begin{equation} \lb{1.4}
\alpha_n = -\ol{\Phi_{n+1}(0)}
\end{equation}
are the Verblunsky coefficients. \eqref{1.3} is replaced by
\begin{equation} \lb{1.5}
d\mu(\theta) = w(\theta)\, \f{d\theta}{2\pi} + d\mu_\s(\theta)
\end{equation}
We have $\abs{\alpha_n} <1$ and $\rho_n$ is defined by
\begin{equation} \lb{1.6}
\rho_n = (1-\abs{\alpha_n}^2)^{1/2}
\end{equation}
For background on OPRL, see \cite{Szb,Chi,FrB,Rice}, and for OPUC, see \cite{Szb,GBk,OPUC1,OPUC2}.

Our starting point is Szeg\H{o}'s theorem in Verblunsky's form (see Ch.~2 of \cite{OPUC1} for history
and proof):

\begin{theorem}\lb{T1.1} Consider OPUC. The following are equivalent:
\begin{alignat}{2}
&\text{\rm{(a)}} \qquad && \int \log(w(\theta))\, \f{d\theta}{2\pi} >-\infty \lb{1.7} \\
&\text{\rm{(b)}} \qquad && \sum_{n=0}^\infty\, \abs{\alpha_n}^2 <\infty \lb{1.8} \\
&\text{\rm{(c)}} \qquad && \prod_{n=0}^\infty \rho_n >0 \lb{1.9}
\end{alignat}
\end{theorem}

Of course, (b) $\Leftrightarrow$ (c) is trivial and (c) is not normally included. We include it
because for OPRL, (a) $\Leftrightarrow$ (c) and (a) $\Leftrightarrow$ (b) have different analogs.
The analog of (a) $\Leftrightarrow$ (c) for OPRL on $[-2,2]$, which we will call Szeg\H{o}'s
theorem for $[-2,2]$, is:

\begin{theorem}\lb{T1.2} Let $J$ be a Jacobi matrix with $\sigma_\ess (J)=[-2,2]$ and eigenvalues
$\{E_j\}_{j=1}^N$ in $\sigma(J)\setminus [-2,2]$.
Suppose that
\begin{equation} \lb{1.10}
\sum_{j=1}^N (\abs{E_j}-2)^{1/2} <\infty
\end{equation}
Then the following are equivalent:
\begin{alignat}{2}
&\text{\rm{(i)}} \qquad && \int_{-2}^2 (4-x^2)^{-1/2} \log(w(x))\, dx > -\infty \lb{1.11} \\
&\text{\rm{(ii)}} \qquad && \limsup a_1 \dots a_n >0 \lb{1.12}
\end{alignat}
If these hold, then
\begin{equation} \lb{1.13}
\lim_{n\to\infty}\, a_1 \dots a_n
\end{equation}
exists in $(0,\infty)$.
\end{theorem}

\begin{remarks} 1. For a proof and history, see Sect.~13.8 of \cite{OPUC2}\

\smallskip

2. The number of eigenvalues, $N$, can be zero, finite, or infinite.

\smallskip

3. There are also results that imply \eqref{1.10}. For example, if
\eqref{1.11} holds, and the $\limsup$ in \eqref{1.12} is finite,
then \eqref{1.10} holds.

\smallskip
4. \eqref{1.12} involves $\limsup$, not $\liminf$; its converse is
that $a_1\dots a_n\to 0$.
\end{remarks}

The analog of (a) $\Leftrightarrow$ (b) is the following result of Killip--Simon \cite{KS}:

\begin{theorem}\lb{T1.3} Let $J$ be a Jacobi matrix with $\sigma_\ess (J)=[-2,2]$ and
eigenvalues $\{E_j\}_{j=1}^N$ in $\sigma(J)\setminus [-2,2]$. Then
\begin{equation} \lb{1.14}
\sum_{n=1}^\infty b_n^2 + (a_n-1)^2 <\infty
\end{equation}
if and only if the following both hold:
\begin{alignat}{2}
&\text{\rm{(i)}} \qquad && \sum_{j=1}^N (\abs{E_j}-2)^{3/2} <\infty \lb{1.15} \\
&\text{\rm{(ii)}} \qquad && \int_{-2}^2 (4-x^2)^{1/2} \log(w(x))\, dx > \infty \lb{1.16}
\end{alignat}
\end{theorem}

The last two theorems involve perturbations of the Jacobi matrix with $b_n\equiv 0$, $a_n\equiv 1$,
essentially up to scaling and translation, constant $b_n,a_n$. The next simplest situation is perturbations
of periodic Jacobi matrices, that is, $J_0$ has Jacobi parameters $\{a_n^{(0)}, b_n^{(0)}\}_{n=1}^\infty$
obeying
\begin{equation} \lb{1.17}
a_{n+p}^{(0)} =a_n^{(0)} \qquad b_{n+p}^{(0)} = b_n^{(0)}
\end{equation}
for some fixed $p$ and all $n=1,2,\dots$. In that case, we have a set
\[
\calE=\bigcup_{j=1}^{\ell+1} \calE_j
\]
where $\{\calE_j\}_{j=1}^{\ell+1}$ are $\ell+1$ disjoint closed
intervals
\begin{gather*}
\calE_j = [\alpha_j,\beta_j] \\
\alpha_1 < \beta_1 < \alpha_2 < \beta_2 < \cdots < \alpha_{\ell+1} < \beta_{\ell+1}
\end{gather*}
with $\ell$ gaps $(\beta_1,\alpha_2), \dots, (\beta_\ell,\alpha_{\ell+1})$, and
\begin{equation} \lb{1.18}
\sigma_\ess(J_0)=\calE
\end{equation}

We always have $\ell+1\leq p$ and generically $\ell+1=p$. In this
generic case, we say ``all gaps are open." We use $\ell$, the number
of gaps, because $J_0$ is not the only periodic Jacobi matrix
obeying \eqref{1.18}---there is an $\ell$-dimensional manifold,
$\calT_\calE$, of periodic $J_0$'s obeying \eqref{1.18}. Indeed, the
collection of all $\{a_j^{(0)}, b_j^{(0)}\}_{j=1}^p \subset
[(0,\infty)\times\bbR]^p$ obeying \eqref{1.18} for fixed $\calE$ is
an $\ell$-dimensional torus, so $\calT_\calE$ is called the
isospectral torus; see \cite[Chap.~5]{Rice}. That the key to
extending Theorems~\ref{T1.2} and \ref{T1.3} to the periodic case is
an approach to an isospectral torus is an idea of Simon
\cite{OPUC2}.

Damanik, Killip, and Simon \cite{DKS2007} have proven the following analogs of Theorems~\ref{T1.2} and
\ref{T1.3}:
\begin{theorem}\lb{T1.4} Let $\calE$ be the essential spectrum of a periodic $J_0$ and let $J$ be a Jacobi matrix with
\[
\sigma_\ess (J)=\calE
\]
Let $\{E_j\}_{j=1}^N$ be the eigenvalues of $J$ in
$\sigma(J)\setminus \calE$. Suppose that
\begin{equation} \lb{1.19}
\sum_{j=1}^N \dist (E_j,\calE)^{1/2} <\infty
\end{equation}
Then the following are equivalent:
\begin{alignat}{2}
&\text{\rm{(i)}}  \qquad && \int_\calE \dist(x,\bbR\setminus \calE)^{-1/2} \log(w(x))\,dx >-\infty \lb{1.20} \\
&\text{\rm{(ii)}} \qquad && \limsup\, \f{a_1\dots a_n}{C(\calE)^n}
>0  \lb{1.21}
\end{alignat}
\end{theorem}

\begin{remarks} 1. In \eqref{1.21}, $C(\calE)$ is the logarithmic capacity of $\calE$; see \cite{Land,Ran,EqMC} for
a discussion of potential theory.

\smallskip
2. Damanik--Killip--Simon \cite{DKS2007} do not use \eqref{1.21} but instead
\[
\limsup\, \f{a_1\dots a_n}{a_1^{(0)}\dots a_n^{(0)}} >0
\]
Since $a_1^{(0)} \dots a_p^{(0)}=C(\calE)^p$, this is equivalent.
\end{remarks}

\begin{theorem}\lb{T1.5} Let $J_0$ be a periodic Jacobi matrix with all gaps open and essential spectrum
$\calE$. Let $J$ be a Jacobi matrix with
\[
\sigma_\ess (J)=\calE
\]
Let $\{E_j\}_{j=1}^N$ be the eigenvalues of $J$ in
$\sigma(J)\setminus \calE$. Define
\begin{equation} \lb{1.22}
d_m(\{a_n,b_n\}_{n=1}^\infty, \{a'_n,b'_n\}_{n=1}^\infty) =
\sum_{j=0}^\infty e^{-j} [\abs{a_{m+j}-a'_{m+j}} + \abs{b_{m+j}-b'_{m+j}}]
\end{equation}
and
\begin{equation} \lb{1.23}
d_m (\{a_n,b_n\},\calT_\calE) =\min_{(a',b')\subset\calT_\calE}\,
d_m (\{a_n,b_n\}, \{a'_n,b'_n\})
\end{equation}
Then
\[
\sum_{m=1}^\infty d_m (\{a_n,b_n\},\calT_\calE)^2 <\infty
\]
if and only if
\begin{alignat}{2}
&\text{\rm{(i)}} \qquad && \sum_{j=1}^N \dist (E_j,\calE)^{3/2} <\infty \lb{1.24} \\
&\text{\rm{(ii)}} \qquad && \int_\calE \dist(x, \bbR\setminus
\calE)^{1/2} \log(w(x))\,dx >-\infty \lb{1.25}
\end{alignat}
\end{theorem}

While these last two theorems are fairly complete from the point of
view of perturbations of periodic Jacobi matrices, they are
incomplete from the point of view of sets $\calE$. By harmonic
measure on $\calE$, we mean the potential theoretic equilibrium
measure. It is known (Aptekarev \cite{Apt}; see also
\cite{Pe93,Totik01,Rice}) that
\begin{SL}
\item[(i)] $\calE$ is the essential spectrum of a periodic Jacobi matrix if and only if the harmonic measure
of each $\calE_j$ is rational. Theorem~\ref{T1.4} is limited to this
case.
\item[(ii)] All gaps are open if and only if each $\calE_j$ has harmonic measure $1/p$. Theorem~\ref{T1.5}
is limited to this case.
\end{SL}

Our major focus in this work is what happens for a general finite
gap set $\calE$ in which the harmonic measures are not necessarily
rational. This is an announcement. We plan at least two fuller
papers: one \cite{CSZ1} on the structure of the isospectral torus
and one \cite{CSZ2} on Szeg\H{o}'s theorem.

\section{Main Results} \lb{s2}

There are two main results in \cite{CSZ2}. The following is partly new:

\begin{theorem}\lb{T2.1} Suppose $\calE$ is an arbitrary finite gap set
\begin{gather*}
\calE = \bigcup_{j=1}^{\ell+1}\, [\alpha_j,\beta_j] \\
\alpha_1 < \beta_1 < \alpha_2 < \cdots < \beta_{\ell+1}
\end{gather*}
Let $J$ be a Jacobi matrix with
\begin{equation} \lb{2.1a}
\sigma_\ess (J)=\calE
\end{equation}
and let $\{E_j\}_{j=1}^N$ be the eigenvalues of $J$ in
$\sigma(J)\setminus \calE$. Suppose that
\begin{equation} \lb{2.1}
\sum_{j=1}^N \dist (E_j,\calE)^{1/2} <\infty
\end{equation}
Then the following are equivalent:
\begin{alignat}{2}
&\text{\rm{(i)}}  \qquad && \int_\calE \dist(x,\bbR\setminus \calE)^{-1/2} \log(w(x))\,dx >-\infty \lb{2.2} \\
&\text{\rm{(ii)}} \qquad && \limsup\, \f{a_1\dots a_n}{C(\calE)^n}
>0  \lb{2.3}
\end{alignat}
\end{theorem}

That (i) $+$ \eqref{2.1} $\Rightarrow$ (ii) is not new. When $N=0$ (i.e., no bound states),
(i) $\Rightarrow$ (ii) goes back to Widom \cite{Widom}. Peherstorfer--Yuditskii \cite{PY}
proved (i) $\Rightarrow$ (ii) under a condition on the bound states, which after a query
from Damanik--Killip--Simon, Peherstorfer--Yuditskii improved to \eqref{2.1} and posted
on the arXiv \cite{PYarx}. Thus the new element of Theorem~\ref{T2.1} is the converse
direction (ii) $+$ \eqref{2.1} $\Rightarrow$ (i). It does not seem to us that the ideas
in \cite{Widom,PY} alone will provide that half.

Associated to each such $\calE$ is a natural isospectral torus:
certain almost periodic Jacobi matrices that lie in an
$\ell$-dimensional torus. Although the torus, $\calT_\calE$, has
been studied before (e.g., \cite{Widom} or \cite{SY}), many features
are not explicit in the literature, so we wrote \cite{CSZ1}.

We will need the proper analog of the ``Jost function" for this
situation. It involves the potential theorist's Green's function for
$\calE$, $G_\calE$, the unique function harmonic on $\bbC\setminus
\calE$, with zero boundary values on $\calE$ and with
$G_\calE(z)=\log\abs{z} + O(1)$ near infinity. We let $d\rho_\calE$
be the equilibrium measure for $\calE$ with density $\rho_\calE(x)$
with respect to the Lebesgue measure and define $u(0;J)$ by
\begin{equation} \lb{2.3a}
u(0;J) = \prod_{j=1}^N \exp (-G_\calE(E_j)) \exp\biggl( -\f12
\int_\calE \log \biggl( \f{w(x)}{\rho_\calE(x)}
\biggr)d\rho_\calE(x)\biggr)
\end{equation}
We note that since $\rho_\calE(x)\sim\dist(x,\bbR\setminus
\calE)^{-1/2}$, the Szeg\H{o} condition \eqref{2.2} implies the
convergence of the integral in \eqref{2.3a}, and since on
$\bbR\setminus \calE$, $G_\calE(x)$ vanishes as
$\dist(x,\calE)^{1/2}$ as $x\to \calE$, \eqref{2.1} implies
convergence of the product in \eqref{2.3a}.

The other main result is the following:

\begin{theorem}\lb{T2.2} Suppose $J$ is a Jacobi matrix obeying the conditions
\eqref{2.1a}--\eqref{2.3} in $\calE$. Then there is a point
$J_\infty = \{a_n^{(\infty)}, b_n^{(\infty)}\}_{n=1}^\infty
\in\calT_\calE$ so
\begin{equation} \lb{2.4}
\abs{a_n-a_n^{(\infty)}} + \abs{b_n-b_n^{(\infty)}} \to 0
\end{equation}
as $n\to\infty$. Moreover, $a_1\dots a_n/ C(\calE)^n$ is almost
periodic. Indeed,
\begin{equation} \lb{2.5}
\f{a_1\dots a_n}{a_1^{(\infty)}\dots a_n^{(\infty)}}\to \f{u(0;J_\infty)}{u(0;J)}
\end{equation}
More generally, if $d\mu^{(\infty)}$ is the spectral measure for
$J_\infty$, we have that for $x\in\bbC\setminus \calE$,
\begin{equation} \lb{2.6}
\f{p_n(x,d\mu)}{p_n(x,d\mu^{(\infty)})}
\end{equation}
has a limit.
\end{theorem}

\begin{remarks} 1. It is an interesting calculation to check that \eqref{2.5} holds for $\calE=[-2,2]$
based on the formulas in \cite{KS} (see (1.29)--(1.31) of that paper).

\smallskip
2. The limit in \eqref{2.6} can also be described in terms of a suitable ``Jost function" $u$.
\end{remarks}

When there are no bound states (i.e., $N=0$), this is a result of Widom \cite{Widom}.
Peherstorfer--Yuditskii \cite{PY} found a different proof relying on a machinery of
Sodin--Yuditskii \cite{SY} which allowed some bound states, and their note \cite{PYarx}
extended to \eqref{2.1}. So this theorem is not new---what is new is our proof of it
and the compact form of \eqref{2.5} is new.

One application that Killip--Simon \cite{KS} make of
Theorem~\ref{T1.2} is to prove a conjecture of Nevai \cite{Nev92}
that
\begin{equation} \lb{2.7}
\sum_{n=1}^\infty \, \abs{a_n-1} + \abs{b_n} <\infty
\end{equation}
implies \eqref{1.11}. For \eqref{2.7} implies \eqref{1.12} and a result of Hundertmark--Simon
\cite{HunS} says \eqref{2.7} implies \eqref{1.10}. Damanik--Killip--Simon \cite{DKS2007}
used Theorem~\ref{T1.4} and a matrix version of \cite{HunS} to prove an analog of Nevai's
conjecture for perturbations of periodic Jacobi matrices. This leads us to:

\begin{conjecture}\lb{Con2.3} Suppose $\{a_n^{(\infty)}, b_n^{(\infty)}\}_{n=1}^\infty$ lies in
$\calT_\calE$ and $J$ is a Jacobi matrix obeying
\begin{equation} \lb{2.8}
\sum_{n=1}^\infty\, \abs{a_n-a_n^{(\infty)}} + \abs{b_n-b_n^{(\infty)}} <\infty
\end{equation}
Then the Szeg\H{o} condition, \eqref{2.2}, holds.
\end{conjecture}

The issue is whether \eqref{2.8} implies \eqref{2.1}. That it holds
for the eigenvalues above and below the spectrum is a result of
Frank--Simon--Weidl \cite{FSW}, but it remains unknown for
eigenvalues in the gaps. However, Hundertmark--Simon \cite{HS2007}
showed that if for some $\veps >0$,
\begin{equation} \lb{2.8x}
\sum_{n=1}^\infty\, [\log(n+1)]^{1+\veps} [\abs{a_n - a_n^{(\infty)}} + \abs{b_n - b_n^{(\infty)}}] <\infty
\end{equation}
then \eqref{2.1} holds. Thus, we have a corollary of Theorem~\ref{T2.1}:

\begin{corollary}\lb{T2.4} If \eqref{2.8x} holds for some $\{a_n^{(\infty)}, b_n^{(\infty)}\}\in\calT_\calE$,
then \eqref{2.2} holds.
\end{corollary}

The big open question on which we are working is extending the Killip--Simon theorem (Theorem~\ref{T1.3})
to a general finite gap setting.

\section{Covering Maps and Beardon's Theorem} \lb{s3}

To understand the approach to the proofs we will discuss in this
section and the next, we need to explain the machinery behind the
proofs of Theorems~\ref{T1.2}--\ref{T1.5}. It goes back to the
Szeg\H{o} mapping (\cite{Sz22a}; see \cite[Sect.~13.1]{OPUC2}) of
OPRL problems on $[-2,2]$ to OPUC via $x=2\cos\theta=z+z^{-1}$ if
$z=e^{i\theta}$. It was realized by Peherstorfer--Yuditskii
\cite{PYpams} and Killip--Simon \cite{KS} that while $x=2\cos\theta$
will not work on the level of measures if there are mass points
outside $[-2,2]$, the map
\begin{equation} \lb{3.1}
x(z)=z+z^{-1}
\end{equation}
allows one to drag
\begin{equation} \lb{3.2}
m(x)=\int \f{d\mu(t)}{t-x}
\end{equation}
back to $\bbD$ and use function theory on the disk.

Following Sodin--Yuditskii \cite{SY}, we can do something similar
for finite gap situations. $x(z)$ given by \eqref{3.1} is the unique
analytic map of $\bbD$ to $(\bbC\setminus [-2,2])\cup\{\infty\}$
which is a bijection with $x(0)=\infty$, $\lim_{z\to 0} zx(z)>0$. If
$(\bbC\setminus [-2,2])\cup \{\infty\}$ is replaced by
$(\bbC\setminus \calE)\cup\{\infty\}$, there is no map with these
properties because $(\bbC\setminus \calE)\cup\{\infty\}$ is not
simply connected. Rather, its fundamental group, $\pi_1$, is
isomorphic to $F_\ell$, the free non-abelian group on $\ell$
generators. But if we demand that $x$ be onto and only locally
one-one, there is such a map.

For $(\bbC\setminus \calE)\cup\{\infty\}$ has a universal covering
space which is locally homeomorphic to $(\bbC\setminus
\calE)\cup\{\infty\}$ on which $\pi_1$ acts. This local map can be
used to give a unique holomorphic structure, that is, the universal
cover is a Riemann surface and $\pi_1$ acts as a set of
biholomorphic bijections. The theory of uniformization (see
\cite{FarKra}) implies the cover is the unit disk. Thus:

\begin{theorem}\lb{T3.1} There is a unique holomorphic map of $\bbD$ to $(\bbC\setminus \calE)\cup\{\infty\}$
which is onto, locally one-one, with $x(0)=\infty$ and $\lim_{z\to 0 } zx(z) >0$. Moreover, there is
a group $\Gamma$ of M\"obius maps of $\bbD$ onto $\bbD$ so $\Gamma\cong F_\ell$ and
\[
x(z)=x(w) \Leftrightarrow \exists\, \gamma\in\Gamma \text{ so that } \gamma(z)=w
\]
\end{theorem}

Thus, $x$ is automorphic for $\gamma$, that is, $x\circ\gamma=x$. If
one looks at $x^{-1}[(\bbC\setminus [\alpha_1,
\beta_{\ell+1}])\cup\{\infty\}]$, there is a unique connected
inverse image containing $0$, call it $\bbF$. This is $\bbD$ with
$\ell$ orthodisks (i.e., disks whose boundary is orthogonal to
$\partial\bbD$) removed from the upper half-disk and their symmetric
partners under complex conjugation (see Figure~1: the shaded area is
the inverse image of the lower half-plane).

\begin{center}
\begin{figure}[h]
\includegraphics[scale=0.4]{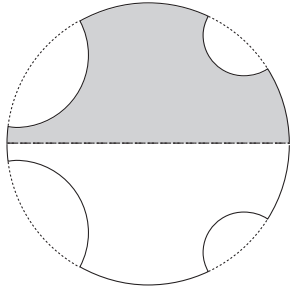}
\caption{The fundamental domain, $\bbF$}
\end{figure}
\end{center}

Label the circles in the upper half-plane $C_1^+, \dots, C_\ell^+$
going clockwise, and $C_1^-, \dots, C_\ell^-$ the conjugate circles.
Let $\gamma_j^\pm$ be the composition of complex conjugation
followed by inversion in $C_j^\pm$, so $\gamma_j^\pm
[\,\ol{\bbF}\,]$ lies inside the disk bounded by $C_j^\pm$. $\Gamma$
consists of words in $\{\gamma_j^\pm\}$, that is, finite products of
these elements with the rule that no $\gamma_j^+$ is next to a
$\gamma_j^-$ (same $j$) for $(\gamma_j^+)^{-1} =\gamma_j^-$. Thus,
$\Gamma =\{\id\}\cup\Gamma^{(1)} \cup \cdots$ where $\Gamma^{(k)}$
has $2\ell(2\ell-1)^{k-1}$ elements, each a word of length $k$.

We define
\begin{equation} \lb{3.3}
\calR_m = \partial\bbD\, \bigg\backslash \,
\bigcup_{\gamma\in\{\id\}\cup\cdots\cup\Gamma^{(m-1)}} \gamma
[\,\ol{\bbF}\,]
\end{equation}
Figure~2 shows three levels of orthocircles. $\calR_3$ is the part of $\partial\bbD$ inside
the 36 small circles.

\begin{center}
\begin{figure}[ht]
\includegraphics[scale=0.35]{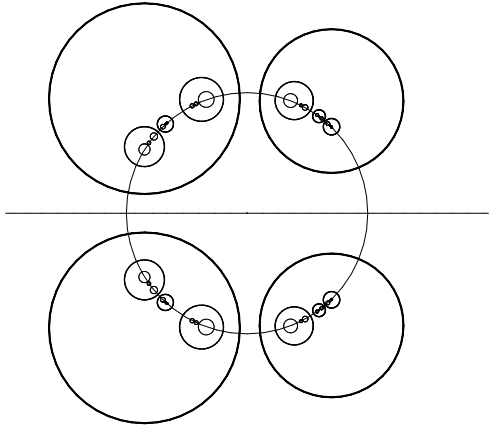}
\caption{Images of $\bbF$ under words of length $\leq 3$}
\end{figure}
\end{center}

In \cite{Bear}, Beardon proved the following theorem:

\begin{theorem}\lb{T3.2} Let $\Gamma$ be a finitely generated Fuchsian group so that the set of
limit points of $\{\gamma(0)\}_{\gamma\in\Gamma}$ is not all of $\partial\bbD$. Then there exists
$t<1$ so that
\begin{equation} \lb{3.4}
\sum_{\gamma\in\Gamma} \, \abs{\gamma'(0)}^t <\infty
\end{equation}
\end{theorem}

The $\Gamma$ associated to $x$ is clearly finitely generated and points in $\ol\bbF\cap\partial\bbD$
are not limit points, so Beardon's theorem applies. (\cite{Rice} has a simple proof of Beardon's
theorem for this special case of interest here.) In \cite{CSZ2}, we show, using some simple
hyperbolic geometry, that \eqref{3.4} implies

\begin{corollary}\lb{C3.3} Let $\abs{\dott}$ be the Lebesgue measure on $\partial\bbD$. Then there exists
$A>0$ and $C$ so that
\begin{equation} \lb{3.5}
\abs{\calR_m} \leq C e^{-Am}
\end{equation}
\end{corollary}

\eqref{3.4} is known to be equivalent to
\begin{equation} \lb{3.6}
\sum_{\gamma\in\Gamma} (1-\abs{\gamma(z)})^t <\infty
\end{equation}
for all $z\in\bbD$. This result for $t=1$ (which goes back to Burnside \cite{Burn}) implies the
existence of the Blaschke product
\begin{equation} \lb{3.7}
B(z,z_0)=\prod_{\gamma\in\Gamma} b(z,\gamma(z_0))
\end{equation}
where
\begin{equation} \lb{3.8}
b(z,w) = -\f{\bar w}{\abs{w}}\, \f{z-w}{1-\bar w z}
\end{equation}
if $w\neq 0$ and $b(z,0)=z$. In particular, we set
\[
B(z)\equiv B(z,z_0=0)
\]

$B$ is related to the Green's function $G_\calE$: we have
\begin{equation} \lb{3.9}
\abs{B(z)} =\exp (-G_\calE(x(z)))
\end{equation}
as can be seen by noting the right side behaves like $C\abs{z}$ near
$z=0$ and \eqref{3.9} holds for $z\in\partial\bbD$.

\section{MH Representation and Szeg\H{o}'s Theorem} \lb{s4}

Simon--Zlato\v s \cite{SZ} and Simon \cite{S288} provided some simplifications of Killip--Simon
\cite{KS} and, in particular, \cite{S288} stated a representation theorem for meromorphic
Herglotz functions. Variants of this representation theorem  are behind parts of \cite{DKS2007}
and other applications of sum rules (e.g., Denisov \cite{Den}).

Our work also depends on such a representation theorem for automorphic meromorphic functions which
obey $\Ima f>0$ on $\bbF\cap\bbC^+$. We prove the following:

\begin{theorem}\lb{T4.1} Let $M(z)=-m(x(z))$, where $m$ is the $m$-function \eqref{3.2} for some $J$,
with $\sigma_\ess(J)=\calE$. For $R<1$, let $B_R(z)$ be the product
$B(z,z_j)$ divided by $B(z,p_j)$ for zeros and poles of $M$ in
$\ol{\bbF}$ with $\Ima z_j\geq 0$, $\Ima p_j \geq 0$ and $\abs{z_j}
<R$, $\abs{p_j}< R$. Then, for $z\in\bbD$,
\begin{equation} \lb{4.1}
B_\infty (z) =\lim_{R\uparrow 1} B_R(z)
\end{equation}
exists for $z$ not a pole of $M$. Moreover, for a.e.\ $\theta\in [0,2\pi)$, $M(e^{i\theta}) =
\lim_{r\uparrow 1} M(re^{i\theta})$ exists,
\begin{equation} \lb{4.2}
\log\abs{M(re^{i\theta})} \in \bigcap_{p<\infty}\, L^p \biggl( \partial\bbD, \, \f{d\theta}{2\pi}\biggr)
\end{equation}
and for $z\in\bbD$,
\begin{equation} \lb{4.3}
a_1 M(z) = B(z) B_\infty(z) \exp\biggl( \f{1}{2\pi} \int \f{e^{i\theta}+z}{e^{i\theta}-z}\,
\log\abs{a_1 M(e^{i\theta})}\, d\theta\biggr)
\end{equation}
\end{theorem}

In proving this, the big difference from the case considered in
\cite{S288} is that there, $\arg M(z)\in (0,\pi)$ in the upper
half-disk. This and a similar estimate for $B_\infty(z)$ prove that
$\arg (M(z)/B(z) B_\infty (z))$ is bounded. Here $\arg M(z)$
is in $(0,\pi)$ only on $\bbF\cap\bbC_+$. In general, if $z\in\gamma
[\bbF]$ where $\gamma$ is a word of length $n$ in $\Gamma$ (written
as a product of generators), then $\abs{\arg M(z)}\leq \pi(2n+1)$.
$\arg(M(z)/B_\infty(z) B(z))$ is {\it not\/} bounded. But by
\eqref{3.5}, the set where $\arg(M(re^{i\theta})/B_\infty
(re^{i\theta}) B(re^{i\theta}))\geq 4\pi (n+1)$ has size (in
$\theta$) bounded by $Ce^{-An}$ uniformly in $r$. This still allows
one to see $\log (M(z)/B(z) B_\infty(z))\in\cap_{p<\infty}
H^p(\bbD)$ and yields \eqref{4.3}.

While there are some tricky points with eigenvalues in gaps, once
one has Theorem~\ref{T4.1}, the proof of Theorem~\ref{T2.1} follows
the strategy used in \cite{Rice} to prove the Szeg\H{o} theorem for
$[-2,2]$. The potential theoretic equilibrium measures enter because
one has:

\begin{proposition}\lb{P4.2} If $f$ is a nice function on $\calE$, then
\begin{equation} \lb{4.4}
\int_{\partial\bbD} f(x(e^{i\theta}))\, \f{d\theta}{2\pi} =
\int_\calE f(x)\, d\rho_\calE(x)
\end{equation}
\end{proposition}

\begin{remark} 1. Since
$\rho_\calE(x)\sim \dist(x,\bbR\setminus \calE)^{-1/2}$, this leads
to Szeg\H{o} conditions like \eqref{2.2}.

\smallskip
2. It is well known how the equilibrium measure is transformed under
conformal mappings (see, e.g., \cite[Prop.~1.6.2]{Fi83}).
\eqref{4.4} is a multi-valued variant of this result.

\smallskip
3. As will be discussed in \cite{CSZ1}, \eqref{3.9} is a special
case of \eqref{4.4}. In fact, one can show that they are actually
equivalent.
\end{remark}

\begin{proof}[Sketch] 1. One proves that
\begin{equation} \lb{4.5}
\abs{B(z)} = \prod_{\gamma\in\Gamma}\, \abs{\gamma(z)}
\end{equation}

\smallskip
2. On $\partial\bbD$, $(\partial\arg\gamma(e^{i\theta})/\partial\theta)>0$, so \eqref{4.5} implies
\begin{equation} \lb{4.6}
\sum_\gamma \, \abs{\gamma'(e^{i\theta})} = \f{d}{d\theta}\, \arg B(e^{i\theta})
\end{equation}

\smallskip
3. This implies
\begin{equation} \lb{4.7}
\int_{\partial\bbD} f(x(e^{i\theta}))\, \f{d\theta}{2\pi} = \int_{\ol{\bbF}\cap\partial\bbD}
f(x(e^{i\theta}))\, \f{d\arg B}{d\theta}\, \f{d\theta}{2\pi}
\end{equation}

\smallskip
4. Since $x$ is two-one from $\ol{\bbF}\cap\partial\bbD$ to $\calE$, this leads to
\begin{equation} \lb{4.8}
\text{LHS of \eqref{4.7}} = \int_\calE f(u)\, \f{d\arg
B(x^{-1}(u))}{du}\, \f{du}{\pi}
\end{equation}

\smallskip
5. By a Cauchy--Riemann equation,
\[
\f{d\arg B(x^{-1}(u))}{du} = \f{\partial\log \abs{B(x^{-1}(u))}}{\partial n}
\]
a normal derivative which is the normal derivative of the Green's function by \eqref{3.9}.

\smallskip
6.
\[
\f{1}{\pi}\, \f{\partial G_\calE}{\partial n}\, (x)\, dx =
d\rho_\calE(x)
\]
completing the proof.
\end{proof}

\section{The Jost Function and Jost Solutions} \lb{s5}

Let $J$ be a Jacobi matrix that obeys the hypotheses of
Theorem~\ref{T2.1}, that is, \eqref{2.1a}, \eqref{2.1}, \eqref{2.2},
and \eqref{2.3} all hold. In that case, we say $J$ is Szeg\H{o} for
$\calE$. For reasons that will become clear shortly, it is useful to
define the {\it Jost function\/} on $\bbD$ by
\begin{equation} \lb{5.1}
u(z,J) = \prod_{j=1}^N B(z,p_j) \exp \biggl( \f{1}{4\pi} \int
\f{e^{i\theta}+z}{e^{i\theta}-z}\, \log \biggl(
\f{\rho_\calE(x(e^{i\theta}))}{w(x(e^{i\theta}))}\biggr)
d\theta\biggr)
\end{equation}
and the {\it Jost solution\/}, $u_n(z,J)$, for $n\geq 0$ by (where $a_0\equiv 1$)
\begin{equation} \lb{5.2}
u_n(z,J) = a_n^{-1} B(z)^n u(z,J^{(n)})
\end{equation}
where $J^{(n)}$ is the $n$ times stripped Jacobi matrix, that is, with Jacobi parameters
$\{a_j^{(n)}, b_j^{(n)}\}$ where
\begin{equation} \lb{5.3}
a_j^{(n)} = a_{j+n} \qquad b_j^{(n)} = b_{j+n}
\end{equation}
Notice because of \eqref{2.1} and \eqref{2.2} the product and integral in \eqref{5.1} converge.
Also notice \eqref{5.1} agrees with \eqref{2.3a} given \eqref{3.9}.
For \eqref{5.2} to make sense, we need:

\begin{proposition}\lb{P5.1} If $J$ is Szeg\H{o} for $\calE$, so is $J^{(n)}$.
\end{proposition}

\begin{proof} It is enough to prove it for $n=1$ and then use induction.
\eqref{2.1a} holds for $J^{(1)}$ by Weyl's theorem and \eqref{2.1} by eigenvalue interlacing.
\eqref{2.3} is trivial for $J^{(1)}$ given it for $J$, and then \eqref{2.2} for $J^{(1)}$
follows from Theorem~\ref{T2.1}.
\end{proof}

Here is the main result about Jost solutions:

\begin{theorem}\lb{T5.2} Let $J$ be Szeg\H{o} for $\calE$. Then {\rm{(}}with $M_n(z)=M(z;J^{(n)})${\rm{)}}
\begin{alignat}{2}
& \text{\rm{(i)}} \qquad && a_{n+1} M_n(z) = B(z)\, \f{u(z, J^{(n+1)})}{u(z,J^{(n)})} \lb{5.4} \\
& \text{\rm{(ii)}} \qquad && a_n M_n(z) = \f{u_{n+1}(z,J)}{u_n (z,J)} \lb{5.5}
\end{alignat}

\begin{SL}
\item[{\rm{(iii)}}] For $z\in\bbD$, $u_n(z,J)$ obeys the difference equation {\rm{(}}$a_0\equiv 1${\rm{)}}
\begin{equation} \lb{5.6}
a_{n-1} u_{n-1} + b_n u_n + a_n u_{n+1} = x(z) u_n
\end{equation}
for $n\geq 1$.

\item[{\rm{(iv)}}] Up to a constant, $u_n (z,J)$ is the unique $\ell^2$ solution of \eqref{5.6}.
\end{SL}
\end{theorem}

\begin{proof}[Sketch] 1. (i) is just a restatement of \eqref{4.3} using the fact that
\begin{align}
a_1^2|M(e^{i\theta})|^2 = \f{\Ima M(e^{i\theta})}{\Ima
M_1(e^{i\theta})}
\end{align}

\smallskip
2. (ii) follows from (i) and the definition \eqref{5.2}.

\smallskip
3. \eqref{5.6} follows from \eqref{5.5} and the coefficient stripping formula for $M$, namely,
\begin{equation} \lb{5.7}
M_n(z)^{-1} = x(z)-b_{n+1} - a_{n+1}^2 M_{n+1}(z)
\end{equation}

\smallskip
4. One proves uniform bounds on $a_n^{-1}$ and $u(z,J^{(n)})$. Since $\abs{B(z)}<1$ on $\bbD$,
$u_n$ goes to zero exponentially and so lies in $\ell^2$. Uniqueness is standard.
\end{proof}

In \cite{CSZ1,CSZ2}, we study boundary values of $u$ as $z\to\partial\bbD$, Green's functions,
and related objects.

\section{Character Automorphic Functions and Asymptotics} \lb{s6}

The key fact in Theorem~\ref{T2.2} is the existence of the limit
point in $\calT_\calE$. The Jost function actually determines the
limit point. To explain how, we need to discuss character
automorphic functions.

If $\gamma$ is a M\"obius transformation of $\bbD$ to $\bbD$ and $b$ is given by \eqref{3.8},
then $h(z)=b(\gamma(z),\gamma(w))$ has magnitude $1$ on $\partial\bbD$ and a zero only at
$z=w$, so $\abs{h(z)}=\abs{b(z,w)}$, but there is generally a nontrivial phase factor
(necessarily constant by analyticity). This implies that for any $w\in\bbD$,
\begin{equation} \lb{6.1}
B(\gamma(z),w) = C_w(\gamma) B(z,w)
\end{equation}
where $\abs{C_w(\gamma)} =1$. Clearly, $C_w(\gamma\gamma') =  C_w(\gamma) C_w(\gamma')$, so
$C_w$ is a character of $\Gamma$, that is, a group homomorphism of $\Gamma$ to $\partial\bbD$.

The set $\Gamma^*$ of such homomorphisms is the dual group of
$\Gamma/[\Gamma,\Gamma]\cong \bbZ^\ell$, so
$\Gamma^*\cong(\partial\bbD)^\ell$ (cf.\ \cite[Chap.~{III}]{SXII}).
Essentially, $C$ is uniquely determined by $C (\gamma_j^+)$,
$j=1,\dots, \ell$.

A meromorphic function on $\bbD$ obeying
\[
f(\gamma(z)) = C(\gamma) f(z)
\]
for all $z\in\bbD$ and $\gamma\in\Gamma$ is called {\it character
automorphic}. \eqref{6.1} says Blaschke products are character
automorphic. One can also see that if $g$ is a real-valued function
on $\calE$, then
\begin{equation} \lb{6.2}
f(z)=\exp\biggl( \int \f{e^{i\theta}+z}{e^{i\theta}-z}\, \log (g(x(e^{i\theta}))\,
\f{d\theta}{2\pi}\biggr)
\end{equation}
is character automorphic, so the Jost function \eqref{5.1} is a product of character
automorphic functions, and so character automorphic. That is, there is a $C_J\in\Gamma^*$
associated with any Szeg\H{o} $J$ via
\begin{equation} \lb{6.3}
u(\gamma(z), J) = C_J(\gamma) u(z,J)
\end{equation}

If $C_0$ is the character associated to the fundamental Blaschke product, $B(z)$, \eqref{5.4}
and the fact that $M$ is automorphic implies
\begin{equation} \lb{6.4}
C_{J^{(n+1)}} = C_{J^{(n)}} C_0^{-1}
\end{equation}
and so
\begin{equation} \lb{6.5}
C_{J^{(n)}} = C_J C_0^{-n}
\end{equation}

A fundamental fact about the map $C$ (discussed in \cite{CSZ1}) is that

\begin{theorem}\lb{T6.1} The map $J\to C_J$ for $J$'s in $\calT_\calE$, from $\calT_\calE$ to $\Gamma^*$,
is a homeomorphism.
\end{theorem}

\begin{corollary}\lb{C6.2} Suppose $J$ is Szeg\H{o} and $J_\infty\in \calT_\calE$ obeys \eqref{2.4}.
Then $J_\infty$ is the unique point in $\calT_\calE$ obeying
\begin{equation} \lb{6.6}
C_{J_\infty} = C_J
\end{equation}
\end{corollary}

\begin{proof}[Sketch] \eqref{2.6} implies that $u(z,J^{(n)})/u(z,J_\infty^{(n)})\to 1$ at points
away from $x^{-1}(\bbR)$ (where it might be $0$), which implies $C_{J^{(n)}}/C_{J_\infty^{(n)}}\to 1$
which, by \eqref{6.5}, implies $C_J/C_{J_\infty}\equiv 1$. Uniqueness follows from the
theorem.
\end{proof}

We have a scheme for proving the convergence result \eqref{2.4} which we hope to implement in
the final version of \cite{CSZ2}. Because it shows a heretofor unknown connection between
Szeg\H{o} behavior and Rakhmanov's theorem, we want to describe the idea.

What can be called the Denisov--Rakhmanov--Remling theorem---namely, a corollary that Remling
\cite{Remppt} gets of his main theorem that extends the theorem of Denisov--Rakhmanov \cite{Denpams}
and Damanik--Killip--Simon \cite{DKS2007} to general finite gap sets---says that any right limit
of a $J$ with $\sigma_\ess(J) =\Sigma_\ac(J) = \calE$ ($\Sigma_\ac$ is the essential support
of the a.c.\ spectrum) lies in $\calT_\calE$. A direct proof of \eqref{6.6} would determine a
unique orbit in $\calT_\calE$ (orbit under coefficient stripping) to which the orbit of $J$ is
asymptotic, and so prove \eqref{2.4}.

We have a proof (whose details need to be checked) that implements this idea and we hope to
use it to get a totally new proof of Theorem~\ref{T2.2} that does not use variational principles.

For now, our proof of Theorem~\ref{T2.2} in \cite{CSZ2}, following Widom \cite{Widom}, uses
the Szeg\H{o} variational approach \cite{Sz20}. In essence, Szeg\H{o} shows $z^n P_n (z+\f{1}{z})$
has a limit $D(0)D(z)^{-1}$ minimizing an $L^2$-norm, subject to taking the value $1$ at $z=0$.
In our case, $B(z)^n P_n (x(z))$ is only character automorphic with an $n$-dependent
character (namely $C_0^n$), so it does not have a fixed limit. Rather, it minimizes an $L^2$-norm
among character automorphic functions (with a fixed but $n$-dependent character)---which explains
why the limiting behavior is only almost periodic.

\bigskip

\end{document}